\documentclass[12pt]{article}

\usepackage{amstext}
\usepackage{amssymb}
\usepackage{amsmath}
\usepackage{pb-diagram} \usepackage{amsthm}
\usepackage{amsfonts}
\usepackage{graphicx}
\usepackage{rotating}
\usepackage{tikz}
\theoremstyle{plain}
 \newtheorem{thm}{Theorem}
 \newtheorem{lem}{Lemma}
 \newtheorem{cor}{Corollary}

\theoremstyle{definition}

\newenvironment{dedication}

\title{$2$-dimensional stratifolds}

\author{J. C. G\'{o}mez-Larra\~{n}aga\thanks{Centro de
Investigaci\'{o}n en Matem\'{a}ticas, A.P. 402, Guanajuato 36000, Gto. M\'{e}xico. jcarlos@cimat.mx} \and F.
Gonz\'alez-Acu\~na\thanks{Instituto de Matem\'aticas, UNAM, 62210 Cuernavaca, Morelos,
M\'{e}xico and Centro de
Investigaci\'{o}n en Matem\'{a}ticas, A.P. 402, Guanajuato 36000,
Gto. M\'{e}xico. fico@math.unam.mx} \and Wolfgang
Heil\thanks{Department of Mathematics, Florida State University,
Tallahasee, FL 32306, USA. heil@math.fsu.edu}}
\date{}

\begin{document}

\maketitle

\begin{dedication}\begin{center} Dedicado a Jos\'e Mar\'ia Montesinos Amilibia con motivo de su 70 aniversario \end{center}       \vspace{\baselineskip}
      \end{dedication}

\begin{abstract} $2$-stratifolds are a generalization of $2$-manifolds in that there are disjoint simple closed curves where several sheets meet. They arise in the study of categorical invariants of $3$-manifolds and may have applications to topological data analysis. We define $2$-stratifolds and study properties of  their associated labeled graphs that determine which ones are simply connected.
\end{abstract}

\begin{section} {Introduction}

Topologically stratified spaces were introduced by Ren\'{e} Thom \cite{T}, who showed that every Whitney stratified space was also a topologically stratified space, with the same strata. Later, Shmuel Weinberger \cite{W} studied the topological classification of stratified spaces. In differential topology the concept of stratifolds, one of the many concepts of stratified spaces, was defined by Mathias Kreck \cite{K}. In topological data analysis [C] one applies topological methods to study complex high dimensional data sets by extracting shapes (patterns) and obtaining insights about them. In these applications, some of these shapes turn out to be simple $2$-dimensional simplicial complexes where it is computationally possible to calculate topological invariants such as the number of connected components or fundamental cycles. It happens that these invariants reflect valuable information from data. Some examples of this methodology to analyze data can be found in \cite{CL} and \cite{CT}.\\

In this paper, we define closed topological $2$-stratifolds and start their study having in mind that these objects could be good models for topological data analysis. A $2$-stratifold $X$ contains a collection of finitely many simple closed curves (the components of the $1$-skeleton $X_1$ of $X$) such that $X-X_1$ is a $2$-manifold and a neighborhood of each component $C$ of $X_1$ consists of sheets (the precise definition is given in section 2). While studying categorical group invariants of $3$-manifolds \cite{GGL}, these $2$-stratifolds appear in a natural way. For instance, if $\mathcal{G}$ is a non-empty family of groups that is closed under subgroups, one would like to determine which (closed) $3$-manifolds have $\mathcal{G}$-category equal to $3$. It can be shown that such manifolds have a decomposition into three compact $3$-submanifolds $H_1 ,H_2 ,H_3$ , where the intersection of $H_i \cap H_j$ (for $i\neq j$) is a compact $2$-manifold, and each $H_i$ is $\mathcal{G}$-contractile (i.e. for each connected component $C$ of $H_i$ the image of the fundamental group of $C$ in the fundamental group of the $3$-manifold is in the family $\mathcal{G}$). The nerve of this decomposition is the union of all the intersections $H_i \cap H_j$ ($i\neq j$). This nerve is a $2$-stratifold and plays an important role in studying categorical invariants. In contrast to $2$-manifolds there is no known classification of $2$-stratifolds in terms of their fundamental group. As a first step in this direction we ask which $2$-stratifolds are $1$-connected and we give a complete solution for the case when the associated graph (see below) is linear. \\

An interesting special class of $2$-complexes, called {\it foams}, has been defined and studied by Khovanov \cite{Ko}.  These were also studied by Carter \cite{SC}, and they are essentially $2$-stratifolds for which a neighborhood of each point of the $1$-skeleton consists of three sheets. We call these {\it trivalent} $2$-stratifolds and develop an algorithm that detects whether such a given $2$-stratifold is simply connected. \\

Here is an outline of the paper: In section 2 we define (closed) $2$-stratifolds $X$ and their associated bicolored graphs $G(X)$, and show that $G(X)$ is a retract of $X$. The graph $G(X)$ consists of ``white" and ``black" vertices that correspond to the $2$-manifold parts and the $1$-skeleton, respectively. Together with a labeling of the edges, $G(X)$ essentially determines $X$. In section 3 we show that the graph of a $1$-connected $2$-stratifold is a tree such that all white vertices have genus $0$ and all terminal vertices are white. Then it is shown that a $1$-connected $2$-stratifold is homotopy equivalent to a wedge of $m{-}n$ $2$-spheres, where $m$ and $n$ are the number of white resp. black vertices of the tree $G(X)$. In section 4 we classify all $1$-connected $2$-stratifolds whose associated graph is linear. Finally, in section 5, we develop an algorithm to decide whether a trivalent $2$-stratifold is simply connected.

\end{section}

\begin{section} {$2$-stratifolds and their Graphs}

A ($2$-layer) $2$-stratifold is a compact, connected Hausdorff space $X$ together with a filtration $\emptyset =X_{0}  \subset X_1  \subset X_2 = X$ by a closed subspace such that $X_1$ is a closed $1$-manifold,  each point $x\in X_1$  has a neighborhood homeomorphic to $\mathbb{R}{\times}CL$, where $CL$ is the open cone on $L$ for some (finite) set $L$ of cardinality $>2$  and
each $x\in  X_2 / X_{1}$ has a neighborhood homeomorphic to $\mathbb{R}^2$ .\\

A component $C\approx S^1$ of $X_1$ has a regular neighborhood $N(C)= N_{\pi}(C):= ( Y {\times}[0,1]) /(y,1)\sim (h(y),0)$. Here  $Y$ is the closed cone on the discrete space $\{1,2,...,d\}$ and $h:Y\to Y$ is a homeomorphism whose restriction to $\{1,2,...,d\}$ is the permutation $\pi:\{1,2,...,d\}\to  \{1,2,...,d\}$. If $\pi '$ is conjugate to $\pi$ in $S_d$, the group of permutations of $\{1,2,...,d\}$, then $N_{\pi}(C) = N_{\pi '}(C)$ . So we may think of $\pi$ as a conjugate class in $S_d$ , that is, as a partition of $d$. A component of $\partial N_{\pi}(C)$ (the set of points having an open neighborhood homeomorphic to $R{\times}[0,1)$) corresponds to a summand of the partition $\pi$.  It follows from Theorems 2 and 3 of \cite{M} that $N(C)= N_{\pi}(C)$ for a unique partition $\pi$.              \\

Another construction of  $N_{\pi}(C)$ is as follows: If $\pi$ is the partition 
$n_1 + n_2 +... + n_r$ of $d$, let $f :{\tilde C} \to C$ be the covering of $C$ where ${\tilde C}$ has components ${\tilde C}_1, {\tilde C}_2,\dots,{\tilde C}_r $ and the restriction of $f$ to ${\tilde C}_i $ is an $n_i$-fold covering $(i=1,...,r)$. Then $N_{\pi}(C)$  is the mapping cylinder of $f$.\\

We take the neighborhoods of the components of $X_1$ sufficiently small so that $N(C)\cap N(C' ) = \emptyset$ if $C' \neq C$. Call  the closures of the components of $N(C)-C$ the {\it sheets} of $N(C)$.\\

 We now construct the connected bipartite graph $G(X,X_1 )$ associated to the 2- stratifold $(X,X_1 )$.\\

Write $M=\overline{X-\cup_j N(C_j)}$ where $C_j$  runs over the components of $X_1$.
The white vertices of the graph $G(X, X_1)$ are the components of $M$; the black vertices are the $N(C_j)$ 's . An edge is a component $S$ of $\partial M$; it joins a white vertex $W$ with a black vertex $N(C_j )$ if $S=W\cap N (C_j)$.\\

We obtain a geometric realization of $G(X,X_1 )$ as an embedding into $X$ as follows: In each component $C_j$ of $X_1$ choose a black vertex $b_j$, in the interior of each component $W_i$ of $M=\overline{X-\cup_j N(C_j)}$ choose a white vertex $w_i$. In each component $S_{ij}$ of $W_i \cap N(C_j )$ choose a point $y_{ij}$, an arc $\alpha_{ij} $ in $W_i$ from $w_i$ to $y_{ij}$ and an arc $\beta_{ij}$ from $y_{ij}$ to $b_j$ in the sheet of $N(C_j )$ containing $y_{ij}$. An edge $e_{ij}$ between $w_i$ and $b_j$ consists of the arc $\alpha_{ij} *\beta_{ij}$. For a fixed $i$, the arcs $\alpha_{ij}$ are chosen to meet only at $w_i$.

\begin{lem}\label{retraction} There is a retraction $r:X\to G(X,X_1 )$.
\end{lem}

\begin{proof} For each component $W_i$ of $M$, the union of the arcs $\bigcup_j \alpha_{ij}$ is a cone on the points $y_{ij}$ with cone point $w_i$, so $\bigcup_j \alpha_{ij}$ can be considered as a retract of $I{\times}I$ and therefore it has the Tietze extension property. Thus the map $\bigcup_j \alpha_{ij} \cup \partial W_i \to \bigcup_j \alpha_{ij}$ that is the identity on $\bigcup_j \alpha_{ij}$ and retracts $\partial W_i$ to $\bigcup \{y_{ij}\}$ extends to a retraction of $W_i$ to $\bigcup_j \alpha_{ij}$. Similarly there is a retraction of $N(C_j )$ to $\bigcup_i \beta_{ij}$. Since the retractions agree on $W_i \cap N (C_j)\to \{y_{ij} \}$ and combine to yield a retraction $r:X\to G(X,X_1 )$.
\end{proof}
 
  We now assign labels to the edges of $G(X,X_1 )$.\\
  
  If $F$ is a compact surface, $g(F)$ will be the genus, with the convention that $g(F)< 0$ if $F$ is nonorientable. Thus $g(P^2) = g(Mb) =  -1, \,\, g(K) = -2$, and so on, where $Mb$ is the Moebius band and $K$ is the Klein bottle. This convention follows Neumann \cite{N}.\\
  
  A white vertex $W$ is labeled with the genus $g$ of $W$ (as defined above). An edge $S$ is labeled by $n$, where $n$ is the summand of the partition $\pi$ corresponding to the component $S$  of  $\partial N_{\pi}(C)$ where $S\subset \partial N_{\pi}(C)$. \\
  
  Note that there is no need to label a black vertex with the partition $\pi$ because the partition is shown by the labels of the adjacent edges. Also note that the the number of boundary components of $W$ is the number of adjacent edges of $W$. If $G=G(X,X_1 )$ is a tree, or if all white vertices have labels $g<0$, then it is easy to see that the labeled graph determines $X$ uniquely. If the graph contains vertices with non-negative genus one needs some additional structure (see \cite{N}): let $G^* $ be the subgraph of $G$ defined by all non-terminal vertices with $g\geq 0$, that is, the union of all edges incident on such vertices and assign a $+$ or  $-$ sign to each edge on $G^* ( X, X_1)$. This information determines a cocyle $\epsilon_G \in H^1 (G^* ; Z/2)$ which assigns to any cycle the number modulo $2$ of $(-)$-edges on that cycle.  This is the additional structure on $G( X, X_1)$ needed to recover $( X, X_1)$.\\

\end{section}

\begin{section} {Simply connected $2$-stratifolds}

Let $X$ be a $2$-stratifold with associated bicolored graph $G=G(X,X_1 )$.\\

By a subcomplex $T$ of $X$ we mean that $T=r^{-1}(\Gamma )$, where $\Gamma$ is a subgraph of $G$ and $r$ is the retraction of Lemma \ref{retraction}.  Let $\hat{T}$  be the quotient of $X$ obtained by collapsing to a point the closure of each component of $X - T$. Note that $\hat{T}$ is a $2$-stratifold whose graph ${\hat \Gamma}$ is is the union of $\Gamma$ and the labeled edges (with their vertices) of $st(\Gamma )-\Gamma$ which are adjacent
to a black vertex of $\Gamma$. (Here $st(\Gamma )$ is the star of $\Gamma$).

\begin{figure}[ht]
\begin{center}
\includegraphics[width=4in]{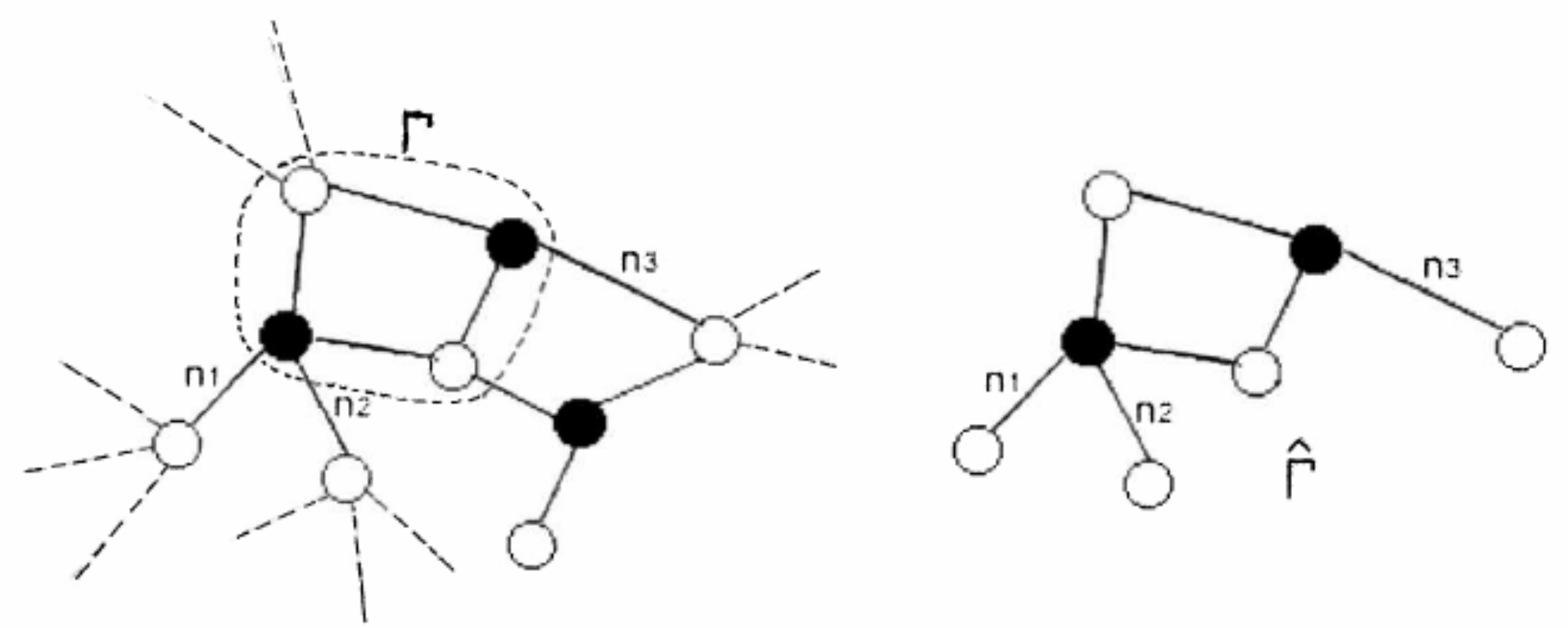} 
\end{center}
\caption{$T$ and ${\hat T}$}
\end{figure}

Then $\pi (\hat{T})$ is a quotient of $\pi (X)$ and similarly $H_1 (\hat{T})$ and $H_1 (\hat{T};\mathbb{Z}_2 )$ are quotients of $H_1 (X)$ and $H_1 (X);\mathbb{Z}_2 )$, respectively. In particular, if $X$ is simply connected, so is $\hat{T}$.

\begin{thm}\label{genus0} (a) If $H_1 (X)$ is finite, then $G$ is a tree.\\
(b) If $H_1 (X,\mathbb{Z}_2 )=0$, then all white vertices of $G$ have genus $0$.\\
(c) If $H_1 (X)= \mathbb{Z}_2^m$ for some $m\geq 0$, then all terminal vertices are white.
\end{thm}

\begin{cor}\label{simplyconnected} If $X$ is simply connected, then $G$ is a tree, all white vertices of $G$ have genus $0$, and all terminal vertices are white.

\end{cor}

\begin{proof} (a) The retraction of Lemma \ref{retraction} induces an epimorphism $r_* :H_1 (X)\to H_1 (G)$ and it follows that $G$ is a tree.\\

(b) If $T$ is a white vertex then $\hat{T}$ is a closed $2$-manifold and since $H_1(X,\mathbb{Z}_2 )=0$, the genus associated to $T$ is $0$.\\

(c) Note that a terminal black vertex $T$ of $G$ corresponds to $\partial N_{\pi}(C)$ (with $\partial N_{\pi}(C)$ connected) and the label $n$ of the edge adjacent to $T$ is the order of $\pi$ (an $n$-cycle) defining $N_{\pi}(C)$, hence $n\geq 3$ and $H_1 (\hat{T})\cong\mathbb{Z}_n$. By the hypothesis this does not occur.
\end{proof}

We have the following homotopy classification of $1$-connected $2$-stratifolds:

\begin{thm} If $X$ is simply connected then  $X$ is homotopy equivalent to a wedge of $m{-}n$ $2$-spheres, where $m$ (resp. $n$) is the number of white (resp. black) vertices of the tree $G$. \end{thm}

In particular, two simply connected $2$-stratifolds have the same homotopy type if and only if they have the same ``deficiency''  $n-m$ and  no $2$-stratifold is contractible.

\begin{proof} To see that $X$ is homotopy equivalent to a wedge of $2$-spheres, we follow Milnor's  proof of  Theorem 6.5  in \cite{Milnor}: Since $\pi_1 (X)=1$ it follows from Hurewicz that $\pi_2 (X)\cong H_2 (X)$. Now $H_2 (X)$ is free abelian (since torsion elements come from $H_3 (X)=0$). The maps $(S^2 ,*)\to (X,*)$ representing the generators of $\pi_2 (X,*)$ combine to a map $f:S^2 \vee \dots \vee S^2 \to X$ that induces an isomorphism $f_* :\pi_2 (S^2 \vee \dots \vee S^2 )\to \pi_2 (X)$. Now it follows from Whitehead that $f$ is a homotopy equivalence.\\
To count the number of spheres in the wedge, recall that $X=M\cup N$, where $M=\overline{X-N}$ and $N=\cup_j N(C_j)$ (and $C_j$ are the components of $X_1$). The Euler characteristic $\kappa (X)=\kappa (M)+\kappa (N)-\kappa (M\cap N) =\beta_0 -\beta_1 +\beta_2$, where $\beta_i$ is the $i$'s Betti number of $X$. Since $\kappa (N)=\kappa (M\cap N)=0$ and $\beta_1 =0$, it follows that $\beta_2 =\kappa (M)-1$, which is also the number of $2$-spheres in $S^2 \vee \dots \vee S^2$. Now $M$ is a (disjoint) union of $m$ punctured spheres and the total number of punctures is the number $e$ of edges of $G$. Therefore $\kappa (M)=2m-e$. Since $G$ is a tree, the number of vertices of $G$ is $m+n=e+1$. It follows that $\beta_2 = m-n$.
\end{proof}
\end{section}

\begin{section} {Linear $2$-stratifolds}

We now consider the $2$-stratifold $X(m_1, n_1, m_2, n_2,\dots, m_r)$   whose associated graph $G=G(m_1, n_1, m_2, n_2,\dots, m_r)$  is the linear graph with successive vertices $w_0, b_1, w_1, b_2, w_2, \dots , b_r, w_r$ and successive labels $m_1, n_1, m_2, n_2, \dots, m_r, n_r$  where $m_i$ (resp. $n_i$) is the label of the edge joining $b_i$ to $w_{i-1}$  (resp. $w_i$) for $i=1, \dots , r$ and all white vertices $w_0, w_1, ... w_r$ have genus $0$.

\begin{figure}[ht]
\begin{center}
\includegraphics[width=4in]{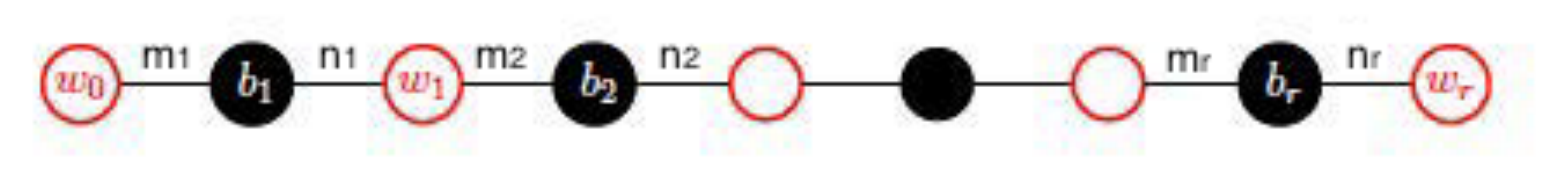} 
\end{center}
\caption{$G(m_1, n_1, m_2, n_2,\dots, m_r)$ }
\end{figure}

The fundamental group of $X$ is the group $$G_{m_1 ,n_1 ,\dots ,m_r ,n_r }=\{x_1, \dots , x_r :\,\, x_1^{m_1}=1, \, x_1^{n_1}=x_{2}^{m_{2}}, \dots ,x_{r-1}^{n_{r-1}}=x_{r}^{m_r}, \, x_r^{n_r}=1 \}$$

\begin{lem}\label{Gr} For integers $m_i$, $n_j$ where $ i=1,\dots r$, $j=1,\dots,r-1$, let $$G_r=\{x_1, \dots , x_r :\,\, x_1^{m_1}=1, \, x_1^{n_1}=x_{2}^{m_{2}}, \dots ,x_{r-1}^{n_{r-1}}=x_{r}^{m_r} \}.$$
If $gcd(m_i ,n_j )=1$ for $i\leq j \leq r-1$, then $G_r\cong \mathbb{Z}_m$, where $m=m_1 \cdots m_r $.
\end{lem}

\begin{proof}First we note that \\
$(*) \qquad\qquad x_i^{m_1 m_2\dots m_i}= 1$ for  $i=1,...,r$\\
This is trivial for $i=1$. By induction,  $1=x_{i-1}^{m_1 \dots m_{i-1}}$, and from $x_{i-1}^{n_{i-1}} =x_i ^{m_i}$ it follows that $1=x_{i-1}^{n_{i-1}m_1 \dots m_{i-1}}=x_i^{m_i m_1 \dots m_{i-1}}$.\\
Since for $i=1,\dots,r-1$, $x_{i}^{n_{i}} \in \langle x_{i+1}\rangle$ (the cyclic subgroup generated by $x_{i+1} $) and $gcd(n_{i},m_1\dots m_{i})=1$, it follows from $(*)$ that $x_{i}\in \langle x_{i+1}\rangle$, hence $G_r$ is cyclic generated by $x_r$. To see that the order of $x_r$ is not less than $m_1 \cdots m_r $, we observe that the relation matrix
$$\begin{pmatrix} 
m_1   &    &&&&&\\
-n_1     &  m_2 &&&&&\\
	& -n_2     &  m_3 &&&&\\ 
&&&	\cdot    &  \cdot &&&\\
&&&&	\cdot    &  \cdot &&\\
&&&&&	-n_{r-1}    &  m_r &\\
\end{pmatrix}$$
can be reduced to a diagonal matrix by elementary row and column operations, which leave the determinant invariant. Since $G_r$ is cyclic, all diagonal entries of the normal form are $=1$ except for the last entry, which is the order of $G_r$. 
\end{proof}

\begin{cor}\label{Gr1} If $gcd(m_i ,n_j )=1$ for $i\leq j \leq r$, then the group $$G_{m_1 ,n_1 ,\dots ,m_r ,n_r }=\{x_1, \dots , x_r :\,\, x_1^{m_1}=1, \, x_1^{n_1}=x_{2}^{m_{2}}, \dots ,x_{r-1}^{n_{r-1}}=x_{r}^{m_r}, \, x_r^{n_r}=1 \}$$ is trivial.
\end{cor}

\begin{proof} By Lemma \ref{Gr}, $G_{m_1 ,n_1 ,\dots ,m_r ,n_r }=\{ x_r :\,\,x_r^{m_1 \cdots m_r }=1,\, x_r^{n_r}=1\}$. Since $gcd(m_1 \cdots m_r ,n_r )=1$, the result follows.
\end{proof}

\begin{lem}\label{Hrabel} Let $H_r$ be the abelian group $$H_r =\{x_1, \dots , x_r :\,\, [x_i ,x_j ]=1, \, x_1^{m_1}=1, \, x_1^{n_1}=x_{2}^{m_{2}}, \dots ,x_{r-1}^{n_{r-1}}=x_{r}^{m_r}, \, x_r^{n_r}=1 \}.$$ If $H_r =1$ then $gcd(m_i ,n_j )=1$ for $1\leq i\leq j \leq r$.
\end{lem}

\begin{proof} If $r=1$,  $H_1 =\{x_1  :\,\,  \, x_1^{m_1}=1, \, x_1^{n_1}=1 \}=1$ implies that $gcd(m_1 ,n_1 )=1$. For $r>1$, if $H_r =1$, then $H_r /\langle x_r \rangle \cong H_{r-1}=1$ and therefore by induction $gcd(m_i ,n_j )=1$ for $i\leq j \leq{r-1}$.  It follows from Lemma \ref{Gr},  that the group $G_r=\{x_1, \dots , x_r :\,\, x_1^{m_1}=1, \, x_1^{n_1}=x_{2}^{m_2},\dots ,x_{r-1}^{n_{r-1}}=x_r^{m_r} \}=\{x_r :\,x_r^{m_1 \dots m_r}= 1\}$.
 Then $H_r =G_r /\langle x_r^{n_r} \rangle =\{x_{r} :\,x_{r}^{m_1 \dots m_r}= 1,\, x_r^{n_r}=1 \} =1$ implies that $gcd(m_1 \dots m_r ,n_r )=1$.
\end{proof}

We know state the main theorem of this section.

\begin{thm}\label{linear} For the  $2$-stratifold  $X=X(m_1, n_1, m_2, n_2,\dots, m_r)$ the following are equivalent:\\
(1) $X$ is simply connected\\
(2) $H_1 (X)=0$\\
(3) $gcd(m_i , n_ j ) = 1$ for $1\leq i\leq j \leq r$.
\end{thm}

\begin{proof} Clearly (1) implies (2). Now $H_1 (X)$ is the group $H_r$ of Lemma \ref{Hrabel}, hence (2) implies (3). Finally (3) implies (1) by Corollary \ref{Gr1}.
\end{proof}

\end{section} 

\begin{section} {Trivalent $2$-stratifolds}

In this section a $2$-stratifold $X$ and its associated labeled bicolored graph $G$ are defined to be {\it trivalent}, if each black vertex $b$ is incident to either three edges each with label $1$ or to two edges, one with label $1$, the other with label $2$, or $b$ is a terminal vertex with adjacent edge of label $3$. This means that a neighborhood of each component $C$ of the $1$-skeleton $X_1$ has $3$ sheets, so the permutation $\pi:\{1,2,3\}\to  \{1,2,3\}$ of the regular neighborhood $N(C)= N_{\pi}(C)$ has partition $1+1+1$ or $1+2$ or $3$.\\

The figure below shows a simple example with one black vertex with partition $1+1+1$ and all terminal edges with label $2$. A simple computation shows that the associated $2$-stratifold has fundamental group $\mathbb{Z}_2$. In fact we show that in general the homology is non-trivial:

\begin{figure}[here]
\begin{center}
\includegraphics[width=2in]{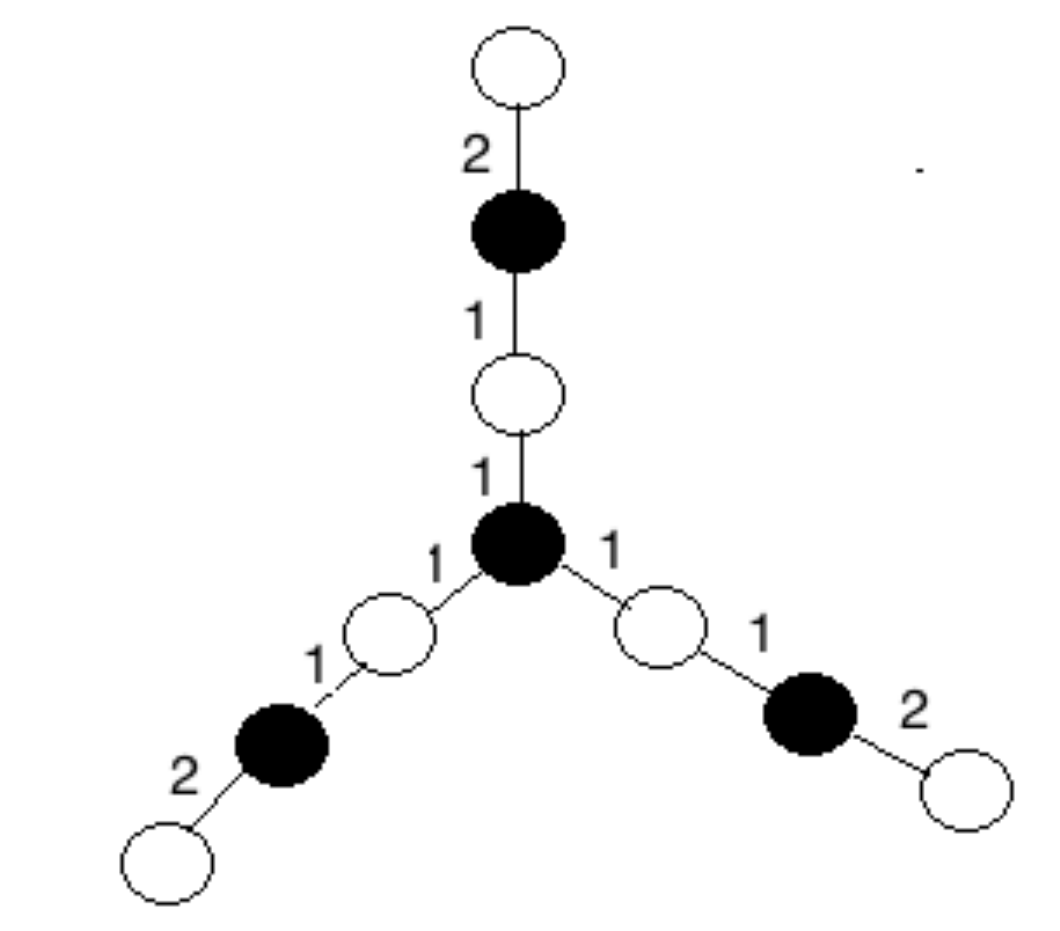} 
\end{center}
\caption{One branch vertex}
\end{figure}

\begin{lem}\label{lemma} If all terminal edges of a trivalent graph $G$ (with a non-zero number of edges) have label $2$, then $H_1 (X,\mathbb{Z}_2)\neq 0$. 
\end{lem}

\begin{proof} By a {\it branch} vertex we mean a vertex of degree $\geq 3$. A {\it terminal branch} is a connected linear subgraph $L$ of $G$ that joins a branch vertex $v$ to a terminal vertex such that no vertex of $L-\{v\}$ is a branch vertex.\\

Note that all terminal vertices of $G$ are white.\\

Suppose the Lemma is false. Let $G$ be a counterexample with a smallest number $n >0$ of edges. If $w$ is a white branch vertex, then for the subcomplex $T=r^{-1}(\Gamma )$, where $\Gamma $  is the complement of a component of $G-\{w\}$, there is a surjection $H_1 (X,\mathbb{Z}_2)\to H_1 ({\hat T},\mathbb{Z}_2)$, which shows that $\Gamma$ is a counterexample to the Lemma with fewer than $n$ edges. Therefore $G$ does not have white branch vertices.\\

We first claim that every terminal branch of $G$ has length $3$ with labels $2, 1, 1$ (starting from the (white) terminal vertex).\\

To see this,  let $w_1 - b_1 - w_2 - b_2 -...- w_r - b_r $ be a terminal branch where $w_1$ is the (white) terminal vertex and $b_r$ is the (black) branch point. We must have $r \geq 2$ since the labels of $w_1 - b_1$ and $b_1 - w_2$ are $2$ and $1$ respectively so $b_1$ cannot be a branch vertex. Suppose the label of the edge $w_2 - b_2$ is $2$. Then, eliminating the edges $w_1 - b_1 - w_2$ from $G$ we obtain a counterexample with $n - 2$ edges. Hence the label of $w_2 - b_2$ is $1$. If $r > 2$ then the label of $b_2 - w_3$ is $2$ and for the linear subgraph $w_1 - b_1 - w_2 - b_2 - w_3$ of $G$ the corresponding sub complex $Y$ has $H_1 (Y,\mathbb{Z}_2 )=0$, which a simple computation shows  is not true. Therefore $r = 2$ and $b_r = b_2$ is a branch vertex with partition $1+1+1$, so the terminal branch $w_1 - b_1 - w_2 - b_2$   has length $3$ and labels $2, 1, 1$. \\

Now, let $b$ be an outermost branch vertex of $G$ (that is a (black) terminal vertex of the tree obtained from $G$ by removing all its terminal branches). If $b$ is the only branch point of $G$ then $G$ is the graph in figure 3, for which the corresponding $2$-stratifold has non-trivial $\mathbb{Z}_2$-homology, a contradiction. In any other case two terminal branches of $G$ have $b$ as an endpoint. Let $\Gamma$ be the graph that is obtained from $G$ by replacing the two terminal branches by a one-edge linear graph with end point $b$ and label $2$. The $2$-stratifold corresponding to the union of the two branches has $\mathbb{Z}_2$ homology $=\mathbb{Z}_2$ and it follows  that the $2$-stratifold $Y$ that corresponds to the graph $\Gamma$ has the same $\mathbb{Z}_2$-homology group as $X$. Thus $\Gamma$ is a counterexample with $n - 5$  edges, contradicting the minimality of $n$. 
\end{proof}

In the proof of the following theorem we give an efficient algorithm to decide whether or not a trivalent $2$-stratifold is $1$-connected.

\begin{thm} There is an algorithm for determining if a trivalent $2$-stratifold $X$  is $1$-connected.
\end{thm}

\begin{proof} By Corollary \ref{simplyconnected}, $G$ must be a tree, all white vertices have genus $0$ and all terminal vertices are white.\\

Step (1) Delete all terminal branches of length $2$ with labels $1, 2$ from $G$ (starting from the (white) terminal vertex).\\

The fundamental group of $X$ is not changed.\\

Step (2) If there is a black branch vertex $b$ with a terminal (white) neighbor, delete $b$ and its edges. \\

This splits $G$ into two subgraphs $\Gamma_1$ and $\Gamma_2$ corresponding to two subcomplexes $Y_1$ and $Y_2$ of $X$ such that $\pi_1 (X)\cong \pi_1 (Y_1 )*\pi_1 (Y_2)$. Then $\pi_1 (X)=1 \iff  \pi_1 (Y_1 )=1$ and $\pi_1 (Y_2)=1$.\\

Repeat steps (1) and (2) as long as possible for the resulting components $Y_i$. Then either some $Y_i$ has all terminal edges of label $2$ or the process yields a collection of white vertices. In the first case $\pi_1 (Y_i)\neq 1$ by Lemma \ref{lemma} and therefore $\pi_1 (X)\neq 1$. In the second case $\pi_1 (X)= 1$.
\end{proof}

We finally give a homology classification of trivalent simply-connected $2$-stratifolds.

\begin{thm} A trivalent 2-stratifold is $1$-connected if and only if $H_1(X;\mathbb{Z}_6) = 0$.
\end{thm}

\begin{proof} Note that the condition is equivalent to $H_1 (X;\mathbb{Z}_2) = 0$ and $H_1 (X;\mathbb{Z}_3) = 0$. The first condition implies by Theorem \ref{genus0}(a),(b) that $G$ is a tree with all white vertices of genus $0$. Since $G$ is trivalent, for a terminal black vertex $T$ we would have a surjection $H_1 (X;\mathbb{Z}_3 )\to H_1 ({\hat T},\mathbb{Z}_3 )\cong \mathbb{Z}_3 $, which is impossible by the second condition. Thus all terminal vertices are white. Following the algorithm, we note that in step (1) the $\mathbb{Z}_2$-homology does not change and in step (2), $H_1 (X;\mathbb{Z}_2) = H_1 (Y_1 ;\mathbb{Z}_2) + H_1 (Y_2 ;\mathbb{Z}_2 )$. Thus we end up with a (possibly disconnected) graph with no edges. This implies that $\pi_1 (X)$ is a free product of trivial groups, that is, $\pi_1 (X)=1$.
\end{proof}

\end{section}

  \vspace{1in}

\end{document}